\documentclass[12pt, reqno]{amsart}
\usepackage{mathrsfs}
\usepackage{graphicx}
\usepackage{tikz}
\usepackage{amsmath}
\usepackage{amsfonts}
\usepackage{amssymb}
\usepackage[pagebackref]{hyperref}
\hypersetup{backref, pagebackref, colorlinks=true}
\usepackage{cite}
\usepackage{color}

\usepackage{epic}

\newtheorem{thm}{Theorem}[section]
\newtheorem{defi}[thm]{Definition}
\newtheorem{lem}[thm]{Lemma}

\newtheorem{remark}[thm]{Remark}

\usepackage{amsthm,amsmath,amssymb,url,cite,color}

\def\red{\textcolor{red}}

\numberwithin{equation}{section}

\def\da{\mathrm{da}}
\def\dd{\mathrm{dd}}

\def\Orb{\mathrm{Orb}}

\def\des{\mathrm{des}}

\def\asc{\mathrm{asc}}
\def\lmin{\mathrm{LRmin}}
\def\rmin{\mathrm{RLmin}}

\def\S{\mathfrak{S}}

\def\NDD{\mathrm{NDD}}
\def\PRW{\mathrm{PRW}}

\def\a{\alpha}

\def\Z{\mathbb Z}

\begin{document}

\title[The Binomial-Stirling-Eulerian Polynomials]{The Binomial-Stirling-Eulerian Polynomials}

\author[ K.Q. Ji]{Kathy Q. Ji}
\address[ Kathy Q. Ji]{Center for Applied Mathematics,
Tianjin University, Tianjin 300072, P.R. China}
\email{kathyji@tju.edu.cn}

\author[Z. Lin]{Zhicong Lin}
\address[Zhicong Lin]{Research Center for Mathematics and Interdisciplinary Sciences, Shandong University \& Frontiers Science Center for Nonlinear Expectations, Ministry of Education, Qingdao 266237, P.R. China}
\email{linz@sdu.edu.cn}

\date{\today}
\begin{abstract}

We introduce the binomial-Stirling-Eulerian polynomials, denoted  $\tilde{A}_n(x,y|{\alpha})$, which encompass binomial coefficients, Eulerian numbers and two Stirling statistics: the left-to-right minima and the right-to-left minima. When $\alpha=1$, these polynomials reduce to the binomial-Eulerian polynomials $\tilde{A}_n(x,y)$, originally named by Shareshian and Wachs and explored by Chung-Graham-Knuth and Postnikov-Reiner-Williams. We investigate the $\gamma$-positivity of $\tilde{A}_n(x,y|{\alpha})$ from two aspects:
\begin{itemize}
\item  firstly by employing the grammatical calculus introduced by Chen;
\item and secondly by constructing a new group action on permutations.
\end{itemize}
These results extend the symmetric Eulerian identity  found by Chung, Graham and Knuth, and the $\gamma$-positivity of $\tilde{A}_n(x,y)$ first demonstrated by Postnikov, Reiner and Williams.
\end{abstract}

\maketitle

\section{Introduction}

The binomial-Eulerian polynomials, named  by Shareshian and Wachs~\cite{sw}, incorporate both binomial coefficients and Eulerian numbers. They originated from Postnikov, Reiner and Williams's work on generalized permutohedra~\cite{prw} and a symmetric Eulerian identity  found  by Chung, Graham and Knuth~\cite{Chung-Graham-Knuth-2010}.

Let $\mathfrak{S}_n$ denote the set of permutations on $[n]:=\{1,\,2,\ldots, n\}.$  An index $i\in[n-1]$ is called a {\em descent} (resp.,~{\em ascent}) of $\pi\in\S_n$ if $\pi_i>\pi_{i+1}$ (resp.,~$\pi_i<\pi_{i+1}$).
The number of permutations in $\S_n$
 with $k$ descents (or $k$ ascents) is known (see~\cite{pet,st}) as the classical {\em Eulerian number} $\left\langle {n \atop k} \right\rangle$.
For fixed positive integers $a$ and $b$, Chung, Graham and Knuth~\cite{Chung-Graham-Knuth-2010} found the following symmetric Eulerian identity
 \begin{equation} \label{e-a-b}
     \sum_{k\geq 0} {a+b \choose k}
     \left\langle {k \atop a-1} \right\rangle
     = \sum_{k\geq 0} {a+b \choose k}
     \left\langle {k \atop b-1} \right\rangle,
 \end{equation}
 where by convention $\left\langle {0 \atop 0} \right\rangle=1$.
 Several generalizations of this identity were subsequently obtained by Chung and Graham~\cite{Chung-Graham-2012}, Han, Lin and Zeng~\cite{Han-Lin-Zeng-2011} and  Lin~\cite{Lin-2013}.

As observed by Shareshian and Wachs~\cite{sw}, the identity \eqref{e-a-b} is equivalent to the symmetry of   the  {\em binomial-Eulerian polynomials}
   \begin{equation}\label{binom-Euler}
    \tilde{A}_{n}(x):=1+x\sum_{m=1}^n{n\choose m}A_m(x),
 \end{equation}
 where $A_m(x)$ are the classical {\em Eulerian polynomials} given by
 \[A_m(x)=\sum_{k=0}^{m-1}\left\langle {m \atop k} \right\rangle x^k.\]
 In other words, if we write $\tilde{A}_{n}(x)=\sum_{k=0}^n\tilde{A}({n},k)x^k$,
then identity~\eqref{e-a-b} is equivalent to $\tilde{A}({n},k)=\tilde{A}({n},n-k)$.

It is known~\cite[Sec.~8.8]{pet} that
the $h$-polynomials of dual permutohedra are the Eulerian polynomials.
Foata and  Sch\"utzenberger~\cite{fs} first  proved  the following $\gamma$-positivity expansion of Eulerian polynomials:
\begin{thm}[Foata-Sch\"utzenberger]\label{FS}
\begin{equation}\label{gam:posi}
A_n(x)=\sum_{k=0}^{\lfloor(n-1)/2\rfloor}\gamma_{n,k}x^i(1+x)^{n-1-2k},
\end{equation}
where $\gamma_{n,k}$ counts the number of permutations $\sigma=\sigma_1\cdots \sigma_n \in \S_n$ with $k$ descents and without double descents {\rm(}see Definition~\ref{four:st} for double descents{\rm )}.
\end{thm}
An elegant combinatorial proof of~\eqref{gam:posi} via a group action was later constructed by Foata and Strehl~\cite{fsh} (see also~\cite{Ath,br,br2,lz}). Since then various refinements of this $\gamma$-positivity expansion, with or without combinatorial proofs, have been found~\cite{br,lin,lz,npt,mp}.

Postnikov, Reiner and Williams~\cite[Section~10.4]{prw} proved that
the $h$-polynomials of dual stellohedra
equal the  binomial-Eulerian polynomials
and provided the combinatorial interpretation
\begin{equation}\label{prw}
\tilde{A}_n(x)=\sum_{\pi\in\PRW_{n+1}}x^{\asc(\pi)},
\end{equation}
where $\PRW_n$ is the set of permutations $\pi\in\S_n$ such that the first ascent (if any) of $\pi$ appears at the letter $1$. For instance,
  \[{\rm PRW}_1=\{1\}, \ {\rm PRW}_2=\{12,21\}, \ {\rm PRW}_3=\{123,132,213, 312,321\},\]
  and
   \begin{align*}
   {\rm PRW}_4=\{&1234,1243,1324,1342,2134,2143,3124,3142,3214,4123,4132,4213,\\[3pt]
   &4312,4321,1423,1432\}.
   \end{align*}
   They also showed the following $\gamma$-positivity expansion of $\tilde{A}_n(x)$, which implies the symmetry and  unimodality of  $\tilde{A}_n(x)$.

\begin{thm}[Postnikov-Reiner-Williams]\label{PRW-thm} For $n\geq 1$,
 \begin{equation}\label{gamma-expan}
 \tilde{A}_n(x)=\sum_{k=0}^{\lfloor\frac{n}{2}\rfloor}
 \tilde{\gamma}_{n,k}x^k(1+x)^{n-2k},
 \end{equation}
 where    $\tilde{\gamma}_{n,k}$ counts the number of permutations $\sigma=\sigma_1\cdots \sigma_{n+1}$ in $\PRW_{n+1}$ with  $k$ ascents and without double ascents {\rm (}see Definition~\ref{four:st} for double ascents{\rm )}.
 \end{thm}

It's worth noting that this result bears an analogy to Theorem \ref{FS}.  Concerning more $\gamma$-positive polynomials arising in enumerative and geometric combinatorics, the interested reader is referred  to the two surveys by Br\"and\'en~\cite{br2} and   Athanasiadis \cite{Ath}, and the book by Petersen~\cite{pet}.

Recently, some related generalizations of Theorem~\ref{PRW-thm} have been found. Two different $q$-analogs of $\tilde A_n(t)$ that possess refined $\gamma$-positivity  were investigated by Shareshian and Wachs~\cite{sw} and Lin, Wang and Zeng~\cite{lwz}. Note that a symmetric polynomial with only real roots implies the $\gamma$-positivity of this polynomial~\cite{br2}.  Haglund and Zhang \cite{Haglund-Zhang-2019} proved that the binomial Eulerian polynomials $\tilde A_n(t)$ are real-rooted, affirming a conjecture by  Ma, Ma and Yeh \cite{Ma-Ma-Yeh-2017}.   Br\"and\'en and Jochemko \cite{Branden-Jochemko-2022}   refines and strengthens the aforementioned results by Haglund and Zhang \cite{Haglund-Zhang-2019}. In particular, by employing symmetric functions,   Shareshian and Wachs provided a new interpretation of $\tilde{\gamma}_{n,k}$ in~\eqref{gamma-expan}.  A combinatorial proof via the Foata--Strehl action was later provided by Lin, Wang and Zeng~\cite{lwz}.

\begin{thm}[Shareshian-Wachs]\label{ci-SW} Suppose that $\tilde{\gamma}_{n,k}$ is the coefficients in the $\gamma$-expansion  \eqref{gamma-expan} of $\tilde{A}_n(x)$. Then $\tilde{\gamma}_{n,k}$ also counts the number of permutations $\sigma=\sigma_1\cdots \sigma_n \in \S_n$ with $k$ descents, in which there does not exist any index $1<i<n$ such that $\sigma_{i-1}>\sigma_i>\sigma_{i+1}$.
\end{thm}

In this paper, we introduce the generalized binomial-Eulerian polynomials $\tilde A_n(t)$, called the binomial-Stirling-Eulerian polynomials,  using two Stirling statistics.  Recall that a  {\em left-to-right  minimum} (resp., {\em right-to-left minimum}) of $\sigma\in\S_n$ is a letter $\sigma_i$ such that $\sigma_{j}>\sigma_i$ for every $j<i$ (resp., $j>i$). Let $\lmin(\pi)$ (resp.,~$\rmin(\pi)$) be the number of left-to-right minima (resp.~right-to-left minimum) of $\pi$.  It is well known~\cite{st} that
$$
\sum_{\pi\in\S_n}\alpha^{\lmin(\pi)}=\sum_{\pi\in\S_n}\alpha^{\rmin(\pi)}=\alpha(\a+1)(\a+2)\cdots(\a+n-1),
$$
which is the generating function for the unsigned {\em Stirling numbers of the first kind}. Thus, the two statistics `$\lmin$' and `$\rmin$' are usually referred to as Stirling statistics on permutations.
The {\em binomial-Stirling Eulerian polynomials}, denoted $\tilde{A}_n(x,y|{\alpha})$, are defined by
\begin{equation*}
\tilde{A}_n(x,y|{\alpha})=\sum_{\sigma \in {\rm PRW}_{n+1}} x^{{\rm des}(\sigma)}y^{{\rm asc}(\sigma)}{\alpha}^{{\rm {LRmin}}(\sigma)+{\rm RLmin}(\sigma)-2}.
\end{equation*}
It should be noted that the {\em Stirling-Eulerian polynomials}
\[A_{n}(x,y|\alpha):=\sum_{\sigma \in \mathfrak{S}_{n+1}} x^{{\rm des}(\sigma)}y^{{\rm asc}(\sigma)}{\alpha}^{{\rm {LRmin}}(\sigma)+{\rm RLmin}(\sigma)-2}\]
 were introduced by Carlitz and Scoville \cite{Carlitz-Scoville-1974} and investigated in further details in the  recent work \cite{Ji-2023} of the first named author.

The main objective of this paper is to investigate  the refined $\gamma$-positivity of $\tilde{A}_n(x,y|{\alpha})$ from both the algebraic and the combinatorial aspects. The first few $\gamma$-positivity expansions of $\tilde A_n(x,y|\a)$  read as
\begin{align*}
\tilde A_1(x,y|\a)&=\a(x+y),\\
\tilde A_2(x,y|\a)&=\a^2(x+y)^2+\a xy,\\
\tilde A_3(x,y|\a)&=\a^3(x+y)^3+(\a+3\a^2)xy(x+y),\\
\tilde A_4(x,y|\a)&=\a^4(x+y)^4+(6\a^3+4\a^2+\a)xy(x+y)^2+(3\a^2+2\a)x^2y^2.
\end{align*}

The following four classical permutation statistics are crucial in our investigation.

\begin{defi}\label{four:st}
For $\sigma\in\S_n$, a value $\sigma_i$ {\rm(}$1\leq i\leq n${\rm )} is called a
{\bf double ascent} {\rm(}resp.,~{\bf double descent, peak, valley}{\rm )} of $\sigma$
if $\sigma_{i-1}<\sigma_i<\pi_{i+1}$ {\rm (}resp.,~$\sigma_{i-1}>\sigma_i>\sigma_{i+1}$,
$\sigma_{i-1}<\sigma_i>\sigma_{i+1}$, $\sigma_{i-1}>\sigma_i<\sigma_{i+1}${\rm )},
where we use the convention $\sigma_0=\sigma_{n+1}=+\infty$.
Let $\da(\sigma)$ {\rm (}resp., $\dd(\sigma)$, ${\rm M}(\sigma)$, ${\rm V}(\sigma)${\rm )} denote the number of double ascents {\rm (}resp., double descents, peaks, valleys{\rm )} of $\sigma$.
\end{defi}

From the algebraic side, we employ the grammatical calculus introduced by Chen \cite{chen} and prove  that  $\tilde A_n(x,y|\a)$ can be expressed  as  the following refinement of the binomial-Stirling-Eulerian polynomials.
 \begin{thm}  \label{mainthm2}
For $n\geq 1$,
 \begin{equation}
\tilde{A}_n(x,y|\alpha)=\sum_{\sigma \in {\rm PRW}_{n+1}} (u_1u_2)^{{\rm M}(\sigma)}u_3^{{\rm da}(\sigma)}u_4^{{\rm dd}(\sigma)}{\alpha}^{{\rm LRmin}(\sigma)+{\rm RLmin}(\sigma)-2},
 \end{equation}
 where $x+y=u_3+u_4$ and $xy=u_1u_2$.
\end{thm}

This result draws a parallel with the following refinement obtained by the first named author \cite{Ji-2023} with the aid of the grammar calculus method.

\begin{thm}{\rm (\cite[Theorem 6.1( $\alpha=\beta$)]{Ji-2023})}\label{Ji:gam}
For $n\geq1$,
$$
A_n(x,y|\a)=\sum_{\sigma\in\S_n} (u_1u_2)^{{\rm M}(\sigma)}u_3^{{\rm da}(\sigma)}u_4^{\dd(\sigma)}\a^{{\rm LRmin}(\sigma)+{\rm RLmin}(\sigma)-2},
$$
where $u_3+u_4=x+y$ and $u_1u_2=xy$.
\end{thm}

It turns out that Theorem~\ref{mainthm2} not only provides three combinatorial interpretations of coefficients in the $\gamma$-expansion of $\tilde{A}_n(x,y|\alpha)$ (see Theorem \ref{PRW-g}) but also can be used to deduce a relation (see Theorem~\ref{mainthm2-des-pk})  that parallels with  the relation between the enumerator of permutations by $({\rm M},\des)$ and the Eulerian polynomials found by Zhuang~\cite{Zhuang-2017}.

From the combinatorial side, we introduce a new group action on permutations in the spirit of the so-called {\em Foata--Strehl action}~\cite{fsh} on permutations. This new group action enables us to prove combinatorially Theorem~\ref{mainthm2}. Furthermore, we found that  this new group action on permutations can also be applied to prove Theorem \ref{Ji:gam}. It should be noted that the Foata--Strehl action proof of a $q$-analog of Theorem~\ref{ci-SW} provided by Lin, Wang and Zeng~\cite{lwz} can not be adopted to prove Theorem~\ref{mainthm2}.

{\bf The rest of this paper is organized as follows.} In Section~\ref{sec:22}, we present  some relevant consequences of Theorem~\ref{mainthm2}. In Section~\ref{sec:2}, we provide the context-free grammar proof of Theorem~\ref{mainthm2}. In Section~\ref{sec:3}, we construct an involution to prove the symmetry of $\tilde{A}_n(x,y|{\alpha})$.
This involution enables us to provide bijective proofs for an $\alpha$-extension of~\eqref{e-a-b} and the equivalence of two combinatorial interpretations given by \eqref{PRW-g-1} and \eqref{PRW-g-3} in Theorem \ref{PRW-g} for the coefficients  in the $\gamma$-expansion of $\tilde{A}_n(x,y|\alpha)$. In Section~\ref{sec:4}, we introduce a  new group action on permutations to prove combinatorially Theorem~\ref{mainthm2} and Theorem~\ref{Ji:gam}.

\section{Relevant consequences of Theorem~\ref{mainthm2}}
\label{sec:22}

Based on Theorem~\ref{mainthm2}, we obtain the following combinatorial interpretations of the coefficients in the $\gamma$-expansion of $\tilde A_n(x,y|\a)$.
Note that when $\alpha=1$ and $y=1$ in Theorem \ref{PRW-g}, we recover Theorem~\ref{PRW-thm}  from \eqref{PRW-g-1}, and similarly,   we obtain Theorem \ref{ci-SW} from \eqref{PRW-g-3}.  In the case of $\alpha=1$, the relation~\eqref{PRW-g-2} bears a resemblance to the one between peak polynomials and Eulerian polynomials, as established by Stembridge~\cite{Stembridge-1997}.

\begin{thm}  \label{PRW-g}
 For $n\geq 1$, let
$$
\tilde A_n(x,y|\a)=\sum_{k=0}^{\lfloor {n/2\rfloor}}\tilde{\gamma}_{n,k}(\alpha)(xy)^{k} (x+y)^{n-2k}.
$$
Then
\begin{align}
\tilde{\gamma}_{n,k}(\alpha)&=\sum_{\sigma \in \Gamma^{(1)}_{n,k}}{\alpha}^{{\rm LRmin}(\sigma)+{\rm RLmin}(\sigma)-2} \label{PRW-g-1}\\[5pt]
&=2^{-n+2k}\sum_{\sigma \in \Gamma^{(2)}_{n,k}}{\alpha}^{{\rm LRmin}(\sigma)+{\rm RLmin}(\sigma)-2} \label{PRW-g-2} \\[5pt]
&=\sum_{\sigma \in \Gamma^{(3)}_{n,k}}{\alpha}^{{\rm RLmin}(\sigma)}, \label{PRW-g-3}
\end{align}
where
\begin{itemize}
\item $\Gamma^{(1)}_{n,k}$ denotes the set of permutations in ${\rm PRW}_{n+1}$ with $k$ ascents and has no double ascents;\\

\item $\Gamma^{(2)}_{n,k}$ denotes the set of permutations in ${\rm PRW}_{n+1}$ with $k$ peaks; \\

 \item $\Gamma^{(3)}_{n,k}$ denotes the set of permutations $\sigma=\sigma_1\cdots \sigma_n \in \S_n$ with $k$ descents, in which there  does not exist any index $1<i<n$ such that $\sigma_{i-1}>\sigma_i>\sigma_{i+1}$.
\end{itemize}

\end{thm}

\begin{proof}

(1) By definition, it is not hard to see that for $\sigma\in \mathfrak{S}_n$,
\begin{equation}\label{rel-aa}
{\rm {M}}(\sigma)+{\rm {dd}}(\sigma)={\rm {des}}(\sigma) \quad \text{and} \quad  {\rm {M}}(\sigma)+{\rm {da}}(\sigma)={\rm {asc}}(\sigma).
\end{equation}
Setting $u_3=0$ in Theorem~\ref{mainthm2}, and using \eqref{rel-aa},  we obtain  \eqref{PRW-g-1}.

(2) When $u_3=u_4=v$ and $u_1=u_2=u$ in Theorem \ref{mainthm2},  we find that
\begin{equation*}
u_3=u_4=\frac{x+y}{2}\quad \text{and} \quad  u_1=u_2=\sqrt{xy},
\end{equation*}
and in this case, Theorem~\ref{mainthm2} can be  reformulated as \eqref{PRW-g-2}.

(3) From \eqref{rel-aa}, we find that
\begin{equation}\label{dd:da}
\dd(\sigma)+\da(\sigma)=n-1-2{\rm M}(\sigma).
\end{equation}
Notice that if $\dd(\sigma)=0$, then
\begin{equation}\label{M:des}
{\rm M}(\sigma)=\des(\sigma).
\end{equation}
 Setting $u_4=0$ in Theorem~\ref{mainthm2} and using~\eqref{dd:da} yields
$$
\tilde A_n(x,y|\a)=\sum_{\substack{\sigma \in {\rm PRW}_{n+1}\\[2pt]\dd(\sigma)=0}}(xy)^{{\rm des}(\sigma)} (x+y)^{n-2{\rm des}(\sigma)}{\alpha}^{{\rm LRmin}(\sigma)+{\rm RLmin}(\sigma)-2}.
$$
Let $\NDD_n$ denote the set of   permutations $\sigma=\sigma_1\cdots \sigma_n \in \S_n$, in which there does not exist any index $1<i<n$ such that $\sigma_{i-1}>\sigma_i>\sigma_{i+1}$. As the first letter of each $\sigma\in\PRW_{n+1}$ with $\dd(\sigma)=0$ must be $1$, the expansion~\eqref{PRW-g-3} then follows from  the simple one-to-one correspondence
$$
\sigma_1\sigma_2\cdots\sigma_{n+1}\mapsto (\sigma_2-1)\cdots(\sigma_n-1)
$$
between $\{\sigma\in\PRW_{n+1}:\dd(\sigma)=0\}$ and $\NDD_n$.
\end{proof}

When set $u_1=uv$, $u_2=u_3=w$, and $u_4=v$ in Theorem \ref{mainthm2}, we arrive at the following relation, which resembles the relationship between joint polynomials involving peaks and descents and the Eulerian polynomials established by Zhuang~\cite{Zhuang-2017}.

\begin{thm} \label{mainthm2-des-pk} For $n\geq 1$,
\begin{align*} \label{main2-cora-2}
&\sum_{\sigma \in {\rm PRW}_{n+1}}u^{{\rm M}(\sigma)}v^{{\rm des}(\sigma)}w^{{\rm asc}(\sigma)}{\alpha}^{{\rm LRmin}(\sigma)+{\rm RLmin}(\sigma)-2}\nonumber \\[5pt]
&=\sum_{\sigma \in {\rm PRW}_{n+1}}x^{{\rm des}(\sigma)}y^{{\rm asc}(\sigma)}{\alpha}^{{\rm LRmin}(\sigma)+{\rm RLmin}(\sigma)-2},
\end{align*}
where
\begin{equation*}\label{defi-x}
x=\frac{(w+v)-\sqrt{(w+v)^2-4uvw}}{2}\quad\text{and}\quad y=\frac{(w+v)+\sqrt{(w+v)^2-4uvw}}{2}.
\end{equation*}
\end{thm}

We conclude this section with an immediate consequence of Theorem~\ref{PRW-g}, which provides an interpretation of the $\a$-extension of secant  numbers. Recall that a permutation $\pi\in\S_n$ is {\em alternating} (or {\em down-up}) if
$$
\pi_1>\pi_2<\pi_3>\pi_4<\pi_5>\cdots.
$$
Let $\mathcal{A}_n$ be the set of all alternating permutations in $\S_n$. Recall that the {\em Euler numbers} $\{E_n\}_{n\geq0}$   are defined by
 $$
 \sec(x)+\tan(x)=\sum_{n\geq0} E_n\frac{x^n}{n!}.
 $$
 The numbers $E_{n}$ with even indices and odd indices  are known as {\em secant numbers} and {\em tangent numbers}, respectively. A classical result due to Andr\'e~\cite{And} asserts that $|\mathcal{A}_n|=E_n$.

  Setting $x=-1$ and $y=1$ in Theorem~\ref{PRW-g}, we obtain the following interpretation of the $\a$-extension of secant  numbers.
 \begin{thm}\label{cor1}
 For $n\geq1$,
 $$
 \tilde A_n(-1,1|\a)=
 \begin{cases}
 (-1)^{n/2}\sum_{\sigma\in\mathcal{A}_n}\a^{\rmin(\sigma)}, &\text{if $n$ is even;}\\
 0, &\text{if $n$ is odd.}
 \end{cases}
 $$
 \end{thm}
It should be noted that the first named author ~\cite[Thm.~1.12]{Ji-2023} also found an interpretation of the $\a$-extension of secant  numbers.
\begin{thm}{\rm (\!\!\cite[Thm.~1.12]{Ji-2023})} \label{cor1}
 For $n\geq1$,
 $$
  A_{n+1}(-1,1|\a/2)=
 \begin{cases}
 (-1)^{n/2}\sum_{\sigma\in\mathcal{A}_n}\a^{\rmin(\sigma)}, &\text{if $n$ is even;}\\
 0, &\text{if $n$ is odd.}
 \end{cases}
 $$
 \end{thm}
 Combining with Theorem~\ref{cor1} results in the following connection between these two kinds  of Stirling-Eulerian polynomials.

 \begin{thm}
 For $n\geq1$,
 $$
   \tilde A_{n}(-1,1|\alpha)=A_{n+1}(-1,1|\alpha/2).
 $$
 \end{thm}
 \begin{remark}
 It would be interesting to construct a direct bijective proof of the above relationship.
 \end{remark}

\section{Context-free grammar proof of Theorem~\ref{mainthm2} and relevant results}%
\label{sec:2}
In this section, we employ  the context-free grammar approach to  present the proof of Theorem  \ref{mainthm2}. It turns out that the grammar approach played an important role in the study of the $\gamma$-positivity  of the Eulerian polynomials or other combinatorial polynomials, see Chen and Fu \cite{Chen-Fu-2022, Chen-Fu-2022-a}, Chen, Fu and Yan \cite{Chen-Fu-Yan-2023}, Lin, Ma and Zhang~\cite{lmz},  and Ma, Ma and Yeh \cite{Ma-Ma-Yeh-2019}.

A context-free grammar $G$ over a set $V=\{x,y,z,\ldots\}$ of variables is a set of substitution rules replacing a variable in $V$ by a Laurent polynomial of variables in $V$. For a context-free grammar $G$ over $V$, the formal derivative $D$ with respect to $G$ is defined as a linear operator acting on Laurent polynomials with variables in $V$ such that each substitution rule is treated as the common differential rule that satisfies  the following relations:
$$
D(u+v)=D(u)+D(v)\quad\text{and}\quad D(uv)=D(u)v+uD(v).
$$

Let
\[\widetilde{P}_n(u_1,u_2,u_3,u_4|{\alpha})=\sum_{\sigma \in {\rm PRW}_{n+1}} (u_1u_2)^{{\rm M}(\sigma)}u_3^{{\rm da}(\sigma)}u_4^{{\rm dd}(\sigma)}{\alpha}^{{\rm LRmin}(\sigma)+{\rm RLmin}(\sigma)-2}.\]
To prove Theorem  \ref{mainthm2} by employing the grammatical calculus, we are required to establish the grammars for $\tilde{A}_n(x,y|{\alpha})$ and $\widetilde{P}_n(u_1,u_2,u_3,u_4|{\alpha})$. Rather than focusing on these two classes of polynomials, we will explore  two more general polynomials $\tilde{A}_n(x,y,z|{\alpha})$ and $\widetilde{P}_n(u_1,u_2,u_3,u_4,u_5|{\alpha})$. The primary reason for this choice is that the grammars for both $\tilde{A}_n(x,y,z|{\alpha})$ and $\widetilde{P}_n(u_1,u_2,u_3,u_4,u_5|{\alpha})$ can be readily determined.
Let
\begin{equation*}
\tilde{A}_n(x,y,z|{\alpha})=\sum_{\sigma \in {\rm PRW}_{n+1}}  x^{{\rm des}(\sigma)-{\rm {LRmin}}(\sigma)+1}y^{{\rm asc}(\sigma)}z^{{\rm {LRmin}}(\sigma)-1}{\alpha}^{{\rm {LRmin}}(\sigma)+{\rm RLmin}(\sigma)-2}
\end{equation*}
and
\begin{equation*}
\begin{aligned}
&\widetilde{P}_n(u_1,u_2,u_3,u_4,u_5|{\alpha})\\[5pt]
&\quad =\sum_{\sigma \in {\rm PRW}_{n+1}} (u_1u_2)^{{\rm M}(\sigma)}u_3^{{\rm da}(\sigma)}u_4^{{\rm dd}(\sigma)-{\rm LRmin}(\sigma)+1}u_5^{{\rm LRmin}(\sigma)-1}{\alpha}^{{\rm LRmin}(\sigma)+{\rm RLmin}(\sigma)-2}.
\end{aligned}
\end{equation*}
We have the following grammatical interpretations for
$\tilde{A}_n(x,y,z|{\alpha})$ and $\widetilde{P}_n(u_1,u_2,u_3,u_4,u_5|{\alpha})$.

\begin{lem}\label{thm-grammar-euler}
Let  $D_{{G}}$ be the formal derivative with respect to the  following grammar
\begin{equation}\label{gramma-e}
    {G}=\{a\rightarrow a\alpha(z+y),\, x\rightarrow xy,\,y\rightarrow xy  \}.
     \end{equation}
 Then,
\begin{equation*}
    {D^{n}_{{G}}}(a)=a \tilde{A}_n(x,y,z|{\alpha}).
\end{equation*}
\end{lem}

\begin{lem}\label{thm-grammar-euler-exten}
Let  $D_{\widetilde{G}}$ be the formal derivative with respect to the  following grammar
\begin{equation}\label{L-P-four}
    \widetilde{G}=\{a\rightarrow a\alpha(u_3+u_5), \,u_4\rightarrow u_1u_2, \, u_3\rightarrow u_1u_2,\,u_1\rightarrow u_1u_3, u_2\rightarrow  u_2u_4 \}.
     \end{equation}
 Then,
\begin{equation*}
    {D^{n}_{\widetilde{G}}}(a)=a\widetilde{P}_n
    (u_1,u_2,u_3,u_4|{\alpha}).
\end{equation*}
\end{lem}

Using the grammatical calculus, we ultimately demonstrate that
 \begin{thm}  \label{mainthm2-var} For $n\geq 1$,
 \begin{equation*}
\widetilde{P}_n(u_1,u_2,u_3,u_4,u_5|{\alpha})=\tilde{A}_n(x,y,z|\alpha),
 \end{equation*}
 where $x+y=u_3+u_4$, $y+z=u_3+u_5$ and $xy=u_1u_2$.
\end{thm}
Evidently, Theorem~\ref{mainthm2} follows from Theorem~\ref{mainthm2-var} by setting $z=x$ and $u_5=u_4$.
The rest of this section is devoted to the grammatical proof of  Theorem~\ref{mainthm2-var}.

\subsection{The grammatical labelings}
In this subsection, we give  proofs of Lemma \ref{thm-grammar-euler} and Lemma \ref{thm-grammar-euler-exten}  by using the grammatical labeling. The notion of a grammatical labeling was introduced by Chen and Fu \cite{Chen-Fu-2017}.

\noindent {\bf Proof of  Lemma \ref{thm-grammar-euler}.} Let $\sigma=\sigma_1\cdots \sigma_n \in {\rm PRW}_n$ such that  $\sigma_k=1$. For $1\leq i\leq n+1$,  we refer to the position $i$ as the one immediately before $\sigma_i$, whereas the position $n+1$ is meant to be the position after $\sigma_n$.
To define the labeling for $\tilde{A}_{n-1}(x,y, z|{\alpha})$, we patch   $+\infty$ to $\sigma$ at the end of $\sigma$, or equivalently, set $\sigma_{n+1}=+\infty$. Then the label of the position $i$ ($2\leq i\leq n+1$) is given by the following procedure:
\begin{itemize}
\item  if $\sigma_{i-1}$ is a ascent, then label the position $i$ by $y$;

\item if $\sigma_{i-1}$ is a descent and $i>k$, then   label the position $i$ by $x$;

\item if $\sigma_{i-1}$ is a descent and $i\leq k$, then   label the position $i$ by $z$;

\item  If $i=n+1$, then label it by $a$;

\item if $\sigma_i$ is a left-to-right minimum or a right-to-left minimum and $i\neq k$, then label ${\alpha}$ below  $\sigma_i$.
\end{itemize}

Below is an example,
\begin{align*}
  & 7\quad 5\quad 4\quad 1\quad 2\quad 3\quad 9\quad 8\quad 6\\[5pt]
  &{\longrightarrow} \ \begin{array}{ccccccccccccccccccccc}
   &7&z&5&z&4&z&1&y&2&y&3&y&9&x&8&x&6&a&+\infty  \\
  &\alpha&&\alpha&&\alpha&&&&\alpha &&\alpha &&&&&&\alpha &&
\end{array}
\end{align*}
Evidently, the weight of $\sigma$ is given by
\[\omega(\sigma)=ax^{{\rm des}(\sigma)-{\rm {LRmin}}(\sigma)+1}y^{{\rm asc}(\sigma)}z^{{\rm {LRmin}}(\sigma)-1}{\alpha}^{{\rm {LRmin}}(\sigma)+{\rm RLmin}(\sigma)-2}.\]
From the definition of the above labeling, we see that
\begin{equation}
  a  \tilde{A}_n(x,y,z|{\alpha})=\sum_{\sigma \in {\rm PRW}_{n+1}}\omega(\sigma).
\end{equation}
 We proceed to show Lemma  \ref{thm-grammar-euler} by induction on $n$. For $n=0$, the statement is obvious. Assume that this assertion holds for $n-1$. To show that it is valid for $n$, we represent a permutation $\sigma=\sigma_1\cdots \sigma_n$ in ${\rm PRW}_n$ by patching   $+\infty$  at the end of $\sigma$. Assume that $\sigma_k=1$ and $\pi$ is the permutation   obtained from $\sigma$ by inserting the element $n+1$ at the position $i$. To ensure that  $\pi \in {\rm PRW}_{n+1}$, the element $n+1$ could not be inserted at the position with the label  $z$.  We consider the following three cases:

Case 1: If $i$ is labeled by $x$ in the labeling of $\sigma$, then the position $i$ of $\pi$ is labeled by $y$ and the position $i+1$ of $\pi$ is labeled by $x$ in the labeling of $\pi$, so the change of weights of $\pi$ is consistent with the substitution rule $x\rightarrow xy$.
\[\cdots\ \sigma_{i-1} \ x\quad\sigma_i \  \cdots\  \Rightarrow \cdots  \sigma_{i-1}\ y\quad  n+1 \quad x\quad  \sigma_i \  \cdots.\]

Case 2: If  $i$ is labeled by $y$ in the labeling of $\sigma$, then it can be checked that the change of weights is in accordance with the rule $y\rightarrow xy$.
\[\cdots\  \sigma_{i-1}\quad y\quad \sigma_i  \  \cdots\  \Rightarrow\  \cdots \ \sigma_{i-1}\quad y \quad n+1 \quad x \quad  \quad \sigma_{i}  \  \cdots.\]

Case 3: If $i$ is labeled by $a$ in the labeling of $\sigma$, that means $i=n+1$, then   the changes of weights caused by the insertion are coded by the rule $a\rightarrow a\alpha (z+y)$.
\[\begin{array}{cccccc}
\cdots &\sigma_n&   a &  +\infty\\
\cdots&      {\alpha} &   &
\end{array}\  \Rightarrow \  \begin{array}{cccccc}
\cdots &\sigma_n&  y& n+1 & a &  +\infty\\
\cdots&   {\alpha} & & {\alpha} &   &
\end{array}\]
or
\begin{align*}
&\quad \begin{array}{ccccccccccc}
\sigma_{1}&z&\cdots&\sigma_{k-1}&z&1&\cdots &\sigma_n&   a &  +\infty\\
\alpha&&\cdots&\alpha& & &\cdots&      \alpha&   &
\end{array}\nonumber \\[10pt]
&  \Rightarrow \  \begin{array}{cccccccccccc}
n+1&z&\sigma_{1}&z&\cdots&\sigma_{k-1}&z&1&\cdots &\sigma_n&   a &  +\infty\\
\alpha&&\alpha&&\cdots&\alpha& & &\cdots&      \alpha&   &
\end{array}.
\end{align*}
Summing up all the cases shows that this assertion is valid for $n$. This completes the proof of Lemma \ref{thm-grammar-euler}. \qed

We proceed to show that the grammar \eqref{L-P-four}
generates the polynomial $\widetilde{P}_n
    (u_1,u_2,u_3,u_4|{\alpha})$.

\noindent {\bf Proof of  Lemma \ref{thm-grammar-euler-exten}.}   The labeling for $\widetilde{P}_{n-1}(u_1,u_2,u_3,u_4,u_5|{\alpha})$ can be described as follows. Let $\sigma=\sigma_1\cdots \sigma_n \in {\rm PRW}_n$ such that $\sigma_k=1$. We patch   $+\infty$ to $\sigma$ at both ends so that there are $n+1$ positions between two adjacent elements. For $1\leq i\leq n+1$,  recall that the position $i$ is said to be the position immediately before $\sigma_i$, whereas the position $n+1$ is meant to be the position after $\sigma_n$. For $2\leq i\leq n+1$, we label the position $i$ as follows:
\begin{itemize}

 \item If $\sigma_{i}$ is a peak, then label the position $i$  by $u_2$ and label the position $i+1$  by $u_1$;

  \item If $\sigma_{i}$ is a  double ascent, then label the position $i$  by $u_3$;

  \item If $\sigma_{i-1}$ is a   double descent and $i\leq k$, then label the position $i$  by $u_5$;

\item If $\sigma_{i-1}$ is a   double descent and $i> k$, then label the position $i$  by $u_4$;

\item  If $i=n+1$, then label it by $a$;

 \item  if $\sigma_i$ is a left-to-right minimum or a right-to-left minimum and $i\neq k$, then label ${\alpha}$ below  $\sigma_i$.
\end{itemize}
The weight $\omega$ of $\sigma$ is defined to be the product of all the labels. For the permutation $7\, 5\,4\,1\, 2\, 3\,9\,8\, 6$, we have the following labeling:
 \begin{align*}
  & 7\quad 5\quad 4\quad 1\quad 2\quad 3\quad 9\quad 8\quad 6\\[5pt]
  &{\longrightarrow} \ \begin{array}{ccccccccccccccccccccc}
   +\infty&7&u_5&5&u_5&4&u_5&1&u_3&2&u_3&3&u_2&9&u_1&8&u_4&6&a&+\infty  \\
  &\alpha&&\alpha&&\alpha&&&&\alpha &&\alpha &&&&&&\alpha &&
\end{array}
\end{align*}
It is clear to see that the weight of $\sigma$ is given by
\[\omega(\sigma)=a(u_1u_2)^{{\rm M}(\sigma)}u_3^{{\rm da}(\sigma)}u_4^{{\rm dd}(\sigma)-{\rm LRmin}(\sigma)+1}u_5^{{\rm LRmin}(\sigma)-1}{\alpha}^{{\rm LRmin}(\sigma)+{\rm RLmin}(\sigma)-2}.\]
From the definition of the above labeling, we see that
\begin{equation}
  a P_n(u_1,u_2, u_3,u_4,u_5|{\alpha},{\beta})=\sum_{\sigma \in {\rm PRW}_{n+1}}\omega(\sigma).
\end{equation}

We proceed to  show  Lemma \ref{thm-grammar-euler-exten} by  induction on $n$. For $n=0$, the statement is obvious. Assume that this assertion holds for $n-1$. To show that it is valid for $n$, we represent a permutation $\sigma=\sigma_1\cdots \sigma_n$ in ${\rm PRW}_n$ by patching   $+\infty$  at both ends so that there are $n+1$ positions between two adjacent elements. Assume that $\sigma_k=1$ and $\pi$ is permutation  in ${\rm PRW}_{n+1}$ obtained from $\sigma$ by inserting the element $n+1$ at the position $i$. Then the label of the position $i$ could not be $u_5$ since $\pi \in {\rm PRW}_{n+1}$. We consider the following six cases:

Case 1: If $i$ is labeled by $u_3$ in the labeling of $\sigma$, then the position $i$ of $\pi$ is labeled by $u_2$ and the position $i+1$ of $\pi$ is labeled by $u_1$ in the labeling of $\pi$, so the change of weights of $\pi$ is consistent with the substitution rule $u_3\rightarrow u_1u_2$.
\[\cdots\ \sigma_{i-1} \ u_3\quad\sigma_i \  \cdots\  \Rightarrow \cdots  \sigma_{i-1}\ u_2\quad  n+1 \quad u_1\quad  \sigma_i \  \cdots.\]

Case 2: If  $i$ is labeled by $u_4$ in the labeling of $\sigma$, then it can be checked that the change of weights is in accordance with the rule $u_4\rightarrow u_1u_2$.
\[\cdots\  \sigma_{i-1}\quad u_4\quad\sigma_i  \  \cdots\  \Rightarrow\  \cdots \ \sigma_{i-1}\quad u_2 \quad n+1 \quad u_1 \quad  \quad \sigma_{i}  \  \cdots.\]

Case 3: If  $i$ is labeled by $u_1$ in the labeling of $\sigma$, then it can be checked that the change of weights is in accordance with the rule $u_1\rightarrow u_1u_3$.
\[\cdots\    \sigma_{i-2}\quad u_2\quad\sigma_{i-1} \quad u_1 \quad \sigma_i \  \cdots\  \Rightarrow \  \cdots    \sigma_{i-2}\quad u_3\quad\sigma_{i-1} \quad u_2 \quad n+1\quad u_1 \quad \sigma_i \  \cdots\  .\]

Case 4: If  $i$ is labeled by $u_2$ in the labeling of $\sigma$, then it can be checked that the change of weights is in accordance with the rule $u_2\rightarrow u_2u_4$.
\[\cdots\    \sigma_{i-1} \quad u_2 \quad \sigma_i \quad  u_1 \quad  \sigma_{i+1} \  \cdots\  \Rightarrow \  \cdots\    \sigma_{i-1} \quad u_2\quad n+1 \quad u_1 \quad \sigma_i \quad  u_4 \quad  \sigma_{i+1} \  \cdots\ .\]

Case 5: If $i$ is labeled by $a$ in the labeling of $\sigma$, that means $i=n+1$, then   the changes of weights caused by the insertion are coded by the rule $a\rightarrow a\alpha(u_3+u_5)$.
\[\begin{array}{cccccc}
\cdots &\sigma_n&   a &  +\infty\\
\cdots&      {\alpha} &   &
\end{array}\  \Rightarrow \  \begin{array}{cccccc}
\cdots &\sigma_n&  u_3& n+1 & a &  +\infty\\
\cdots&   {\alpha} & & {\alpha} &   &
\end{array}\]
or
\begin{align*}
&\quad \begin{array}{ccccccccccc}
\sigma_{1}&u_5&\cdots&\sigma_{k-1}&u_5&1&\cdots &\sigma_n&   a &  +\infty\\
\alpha&&\cdots&\alpha& & &\cdots&      \alpha&   &
\end{array}\nonumber \\[10pt]
&  \Rightarrow \  \begin{array}{cccccccccccc}
n+1&u_5&\sigma_{1}&u_5&\cdots&\sigma_{k-1}&u_5&1&\cdots &\sigma_n&   a &  +\infty\\
\alpha&&\alpha&&\cdots&\alpha& & &\cdots&      \alpha&   &
\end{array}.
\end{align*}
Summing up all the cases shows that this assertion is valid for $n$. This completes the proof. \qed

\subsection{The grammatical  derivation for Theorem \ref{mainthm2-var} }

With Lemma \ref{thm-grammar-euler} and Lemma \ref{thm-grammar-euler-exten} in hand, we are now in a position to give a proof of Theorem \ref{mainthm2-var} by using the transformations between the grammar \eqref{gramma-e} and the grammar \eqref{L-P-four}.  It should be noted that the idea of transformations of grammars was proposed by Ma, Ma and Yeh \cite{Ma-Ma-Yeh-2019} and further developed by Chen and Fu \cite{Chen-Fu-2022-a}.

\noindent {\bf Proof of  Theorem \ref{mainthm2-var}.} By Lemma \ref{thm-grammar-euler} and Lemma \ref{thm-grammar-euler-exten}, it suffices to show that for $n\geq 1$,
\begin{equation}\label{mainthm2-trans}
 {D^{n}_{{G}}}(a)={D^{n}_{\widetilde{G}}}(a)
\end{equation}
under the assumption that
\begin{equation}\label{condition}
x+y=u_3+u_4, \quad y+z=u_3+u_5 \quad  \text{and} \quad  xy=u_1u_2.
\end{equation}

To begin with, by employing induction,  it is easily established  that
\begin{align}\label{pf-joi-tem-1}
{D^{n}_{{G}}}(a)=\sum_{2i+j+k=n}C_{i,j,k}(\alpha)a(xy)^i(y+z)^j(x+y)^k,
\end{align}
and
\begin{align}\label{pf-joi-tem-1cc}
{D^{n}_{\widetilde{G}}}(a)=\sum_{2i+j+k=n}\widetilde{C}_{i,j,k}(\alpha)a(u_1u_2)^i(u_3+u_5)^j(u_3+u_4)^k,
\end{align}
where   $C_{i,j,k}(\alpha)$  and $\widetilde{C}_{i,j,k}(\alpha)$ are  polynomials in $\alpha$ with integer coefficients.

Under the assumption \eqref{condition},  it is easy to see that
\begin{align} \label{pf-joi-tem-4a}
D_{{G}}(a)&=a\alpha(z+y)=a\alpha(u_3+u_5)=D_{\widetilde{G}}(a),\\[5pt] \label{pf-joi-tem-4}
D_{{G}}(x+y)&=2xy=2u_1u_2=D_{\widetilde{G}}(u_3+u_4),\\[5pt]  \label{pf-joi-tem-3a}
D_{{G}}(y+z)&=xy=u_1u_2=D_{\widetilde{G}}(u_5+u_4),\\[5pt]
D_{{G}}(xy)&=xy(x+y)=u_1u_2(u_3+u_4)=D_{\widetilde{G}}(u_1u_2). \label{pf-joi-tem-3}
\end{align}
We proceed to show \eqref{mainthm2-trans} is valid under the   assumption \eqref{condition} by induction on $n$. By \eqref{pf-joi-tem-4a}, we see that   \eqref{mainthm2-trans} is valid when $n=1$.  Assume that \eqref{mainthm2-trans} holds for $n$ under the   assumption \eqref{condition},  that is, ${D^{n}_{{G}}}(a)={D^{n}_{\widetilde{G}}}(a)$.
It implies that for $0\leq i,j,k\leq n$, where $2i+j+k=n$,
\begin{equation}\label{pf-joi-tem-2}
C_{i,j,k}(\alpha)=\widetilde{C}_{i,j,k}(\alpha).
\end{equation}

To obtain $D^{n+1}_{{G}}(a)$, we apply $D_{{G}}$ to \eqref{pf-joi-tem-1} to derive that
\begin{align}\label{pf-joi-tem-5}
D^{n+1}_{{G}}(a)&=\sum_{2i+j+k=n}C_{i,j,k}(\alpha)D_G(a)(xy)^i(y+z)^j(x+y)^k  \nonumber \\[5pt]
&\quad + \sum_{2i+j+k=n}iC_{i,j,k}(\alpha)a (xy)^{i-1}D_G(xy)(y+z)^j(x+y)^k \nonumber \\[5pt]
&\quad + \sum_{2i+j+k=n}jC_{i,j,k}(\alpha)a (xy)^{i}(y+z)^{j-1}D_G(y+z)(x+y)^k \nonumber \\[5pt]
&\quad + \sum_{2i+j+k=n}kC_{i,j,k}(\alpha)a (xy)^{i}(y+z)^{j}(x+y)^{k-1}D_G(x+y).
\end{align}
Under the assumption \eqref{condition}, and using    \eqref{pf-joi-tem-4a}, \eqref{pf-joi-tem-4}, \eqref{pf-joi-tem-3a}, \eqref{pf-joi-tem-3} and \eqref{pf-joi-tem-2},   we find  that
\begin{align}\label{pf-joi-tem-5}
D^{n+1}_{{G}}(a)&=\sum_{2i+j+k=n}\widetilde{C}_{i,j,k}(\alpha)D_{\widetilde{G}}(a)(u_1u_2)^i(u_3+u_5)^j(u_3+u_4)^k  \nonumber \\[5pt]
&\quad + \sum_{2i+j+k=n}i\widetilde{C}_{i,j,k}(\alpha)a (u_1u_2)^{i-1}D_{\widetilde{G}}(u_1u_2)(u_3+u_5)^j(u_3+u_4)^k \nonumber \\[5pt]
&\quad + \sum_{2i+j+k=n}j\widetilde{C}_{i,j,k}(\alpha)a (u_1u_2)^{i}(u_3+u_5)^{j-1}D_{\widetilde{G}}(u_3+u_5)(u_3+u_4)^k \nonumber \\[5pt]
&\quad + \sum_{2i+j+k=n}k\widetilde{C}_{i,j,k}(\alpha)a (u_1u_2)^{i}(u_3+u_5)^{j}(u_3+u_4)^{k-1}D_{\widetilde{G}}(u_3+u_4),    \end{align}
which is also equal to $D^{n+1}_{\widetilde{G}}(a) $ when we applying ${D}_{\widetilde{G}}$ to \eqref{pf-joi-tem-1cc}.
Hence \eqref{mainthm2-trans} is  also valid for $n+1$. This  completes the proof. \qed

\section{An involution for the symmetry of $\tilde{A}_n(x,y|{\alpha})$}
\label{sec:3}
This section  aims to define an involution $\phi$ on the set ${\rm PRW}_n$ such that for each permutation $\sigma \in {\rm PRW}_n$ and $\pi=\phi(\sigma)$, we have
\begin{equation}\label{thm1aconda}
{\rm des}(\sigma)={\rm asc}(\pi), \quad
{\rm asc}(\sigma)={\rm des}(\pi), \quad {\rm dd}(\sigma)={\rm da}(\pi), \quad
{\rm da}(\sigma)={\rm dd}(\pi)
\end{equation}
and
\begin{equation}\label{thm1acondb}
{\rm {LRmin}}(\sigma)+{\rm RLmin}(\sigma)={\rm {LRmin}}(\pi)+{\rm RLmin}(\pi).
\end{equation}
Clearly, this involution not only supplies a bijective proof for the symmetry of $\tilde{A}_n(x,y|{\alpha})$ in $x$ and $y$ but also establishes a bijective proof for the equivalence between \eqref{PRW-g-1} and \eqref{PRW-g-3} as presented in Theorem \ref{PRW-g}.

It should be noted that  the symmetry of $\tilde{A}_n(x,y|{\alpha})$ is equivalent to a symmetric Stirling-Eulerian identity. Define
$$
\left\langle {m \atop k} \right\rangle_{\a}:=\sum_{\sigma\in\S_m\atop \asc(\sigma)=k}\a^{\rmin(\sigma)},
$$
where $\left\langle {m \atop k} \right\rangle_{\a}$ is called the {\em Stirling-Eulerian number}.  It is not hard to see that
$$
\tilde{A}_n(1,y|{\alpha})=\a^n+y\sum_{m=1}^n{n\choose m}\a^{n-m}\sum_{k=0}^{m-1}\left\langle {m \atop k} \right\rangle_{\a} y^k.
$$
Hence, the symmetry of $\tilde{A}_n(x,y|{\alpha})$ in $x$ and $y$ is equivalent to the following symmetric Stirling-Eulerian identity, from which we recover Chung-Graham-Knuth's identity \eqref{e-a-b} by taking $\alpha=1$.
\begin{thm} For fixed integers $a,b\geq1$,
 \begin{equation} \label{a:e-a-b}
     \sum_{k\geq 0} \a^{a+b-k} {a+b \choose k}
     \left\langle {k \atop a-1} \right\rangle_{\a}
     = \sum_{k\geq 0} \a^{a+b-k} {a+b \choose k}
     \left\langle {k \atop b-1} \right\rangle_{\a},
 \end{equation}
 where by convention $\left\langle {0 \atop 0} \right\rangle_{\a}=1$.
 \end{thm}
 \begin{remark}
 Note that Chung, Graham and Knuth's bijection~\cite{Chung-Graham-Knuth-2010} does not work for~\eqref{a:e-a-b}.
 \end{remark}

{\bf The construction of $\phi$.}  Let $\sigma=\sigma_1\cdots \sigma_n \in {\rm PRW}_n$ such that $\sigma_k=1$. We first put a bar after each right-to-left minimum of $\sigma$. For the permutation
 \begin{equation}\label{defisig}
 \sigma=5\,4\,1\,2\,7\,3\,6\,10\,9\,8,
\end{equation}
 we have
  \begin{equation*}
\sigma= 5\,4\,1\,|\,2\,|\,7\,3\,|6\,| \,10\,9\,8\,|.
 \end{equation*}
If there is only one element $\sigma_j$ between two bars, we call  $\sigma_j$ isolated. For   $\sigma$ given by  \eqref{defisig},  we see that it has two isolated right-to-left minima, that is, $\sigma_4=2$ and $\sigma_7=6$.  For each non-isolated right-to-left minima of $\sigma$, assume that it looks like
\[|\sigma_{i_1}\,\sigma_{i_1+1}\,\ldots\, \sigma_{i_2-1}\,\sigma_{i_2}|.\]
We then transform them to
\[|\sigma_{i_2-1}\,\sigma_{i_2-2}\,\ldots\, \sigma_{i_1}\,\sigma_{i_2}|\]
We denote the resulting permutation by $\tau$.
For $\sigma$ given by \eqref{defisig},  the permutation $\tau$ corresponding to $\sigma$ will be
$$
\tau= 5\,4\,1\,|\,2\,|\,7\,3\,|\,6\,|\,9\,10\,8\,|.
$$
 Assume that there are $r$ isolated right-to-left minima of $\tau$, which are $\tau_{i_1}<\ldots <\tau_{i_r}$. Observe that $\tau_k=1$, so there are $k$ left-to-right minima, which are $\tau_1>\cdots>\tau_{k-1}>\tau_k=1.$ We next remove isolated right-to-left  minima   and left-to-right minima expect $\tau_k=1$   from $\tau$ and put $\tau_{i_r}>\tau_{i_{r-1}}>\ldots>\tau_{i_1}$ before $\tau_k=1$ and then insert $\tau_1>\cdots>\tau_{k-1}$ into the resulting permutation so that $\tau_i$ becomes  an isolated right-to-left minimum for $1\leq i\leq k-1$. We denote the resulting permutation by $\pi$.
For the permutation $\tau$ above, we obtain
$$
\pi= 6\,2\,1\,|\,7\,3\,|\,4\,|\,5\,|\,9\,10\,8\,|
$$
 It is not hard to check that $\pi$  satisfies~\eqref{thm1aconda} and~\eqref{thm1acondb} and this process is reversible. Furthermore, $\phi$ is an involution on $\PRW_n$.

   Applying the above involution $\phi$, we get the following correspondence defined on the set ${\rm PRW}_4$ interchanging   ``$\des$'' and ``$\asc$''.
   \[\begin{array}{llll}
1\,2\,3\,4\leftrightarrows 4\,3\,2\,1,&
1\,2\,4\,3\leftrightarrows 2\,1\,4\,3, & 1\,3\,2\,4\leftrightarrows 4\,1\,3\,2, & 1\,3\,4\,2\leftrightarrows 1\,4\,3\,2 \\[10pt]
1\,4\,2\,3\leftrightarrows 3\,1\,4\,2,&
2\,1\,3\,4\leftrightarrows 4\,3\,1\,2, & 3\,1\,2\,4\leftrightarrows 4\,2\,1\,3, & 4\,1\,2\,3\leftrightarrows 3\,4\,1\,4.
\end{array}
\]

\section{A new group action on permutations for Stirling-Eulerian polynomials}
 \label{sec:4}
In this section, we introduce a new group action on permutations which enables us to prove combinatorially a unified generalization of Theorems~\ref{mainthm2} and~\ref{Ji:gam}.


\subsection{Introducing a new group action on permutations}
In this subsection, we aim to introduce a new group action on permutations.

For any $x\in[n]$ and $\sigma\in\S_n$,
the {\em$x$-factorization} of $\sigma$ is the partition of $\sigma$ into the form $\sigma=w_1w_2 xw_3w_4$,
where $w_i$'s are intervals of $\sigma$ and $w_2$ (resp.~$w_3$) is the maximal contiguous interval
(possibly empty) immediately to the left (resp.~right) of $x$ whose letters are all greater than $x$.
Following Foata and Strehl~\cite{fsh}, the action $\varphi_x$ is defined by
$$
\varphi_x(\sigma)=w_1w_3xw_2w_4.
$$
For example, if $\sigma=217685439\in\S_{9}$ and $x=5$, then $w_1=21$, $w_2=768$, $w_3=\emptyset$, $w_4=439$ and we get $\varphi_x(\sigma)=215768439$. The map $\varphi_x$ is an involution acting on $\S_n$, and for all $x,y\in[n]$, $\varphi_x$ and $\varphi_y$ commute.

For a double ascent (resp.,~double desent) $\sigma_i$ of $\sigma$, if it happens to be  also a right-to-left (resp., left-to-right) minimum, then it is called a {\em rlmin-double ascent} (resp., {\em lrmin-double descent}); otherwise, it is called an {\em internal  double ascent (resp., internal  double descent)}.
If $x=\sigma_i$ is a rlmin-double ascent (resp., lrmin-double descent) of $\sigma$, then let $\psi_x(\sigma)$ be obtained from $\sigma$ by removing the letter $\sigma_i$ and inserting it immediately before (resp., after) the greatest left-to-right (resp., right-to-left) minimum that is smaller than $\sigma_i$.
 For instance, if $\sigma=6\,10\,8\,3\,1\,4\,9\,2\,\red{\bf5}\,11\,7\in\S_{11}$ and $x=5$ is a rlmin-double ascent, then $3$ is the greatest left-to-right minimum smaller than $5$ and so $\psi_x(\sigma)=6\,10\,8\,\red{\bf5}\,3\,1\,4\,9\,2\,11\,7$. Moreover, if $x$ is neither a rlmin-double ascent nor a lrmin-double descent of $\sigma$, then set $\psi_x(\sigma)=\sigma$.
 It is clear that the map $\psi_x$ is another  involution acting on $\S_n$.

 \begin{figure}
\begin{center}
\begin{tikzpicture}[scale=.5]

 \draw [thick,dotted,red]
 (2.5,7)--(22.5,7) (1,12)--(23.9,12) (11.5,4)--(3.7,4) (23.5,10)--(1.7,10);
 \draw(22.6,7.35) node{\red{\vector(1,0){0.1}}};
 \draw(23.95,12.35) node{\red{\vector(1,0){0.1}}};
  \draw(3.26,4) node{\red{\vector(-1,0){0.1}}};
    \draw(1.6,10) node{\red{\vector(-1,0){0.1}}};

 \draw [thick,dashed]
(5.5,3)--(9.8,3) (7,13)--(8.6,13) (13,9)--(16.8,9) (16,14)--(14,14)
(20.5,8)--(18,8);
 \draw(9.85,3.35) node{\vector(1,0){0.1}};
 \draw(8.7,13.35) node{\vector(1,0){0.1}};
  \draw(16.95,9.35) node{\vector(1,0){0.1}};
    \draw(14.1,14) node{\vector(-1,0){0.1}};
        \draw(18.1,8) node{\vector(-1,0){0.1}};

 \draw(1.15,12) node{\circle*{3.5}};  \draw(0.4,12) node{12};
 \draw(2.65,7) node{\circle*{3.5}};  \draw(2,7) node{7};
  \draw(4.15,1) node{\circle*{3.5}}; \draw(4,0.4) node{1};
  \draw(5.65,3) node{\circle*{3.5}}; \draw(5.2,3.3) node{3};
    \draw(7.15,13) node{\circle*{3.5}}; \draw(6.4,13.2) node{13};
      \draw(8.65,15) node{\circle*{3.5}}; \draw(8.5,15.5) node{15};
        \draw(10.15,2) node{\circle*{3.5}}; \draw(10,1.5) node{2};
          \draw(11.65,4) node{\circle*{3.5}}; \draw(11.9,4.2) node{4};
            \draw(13.15,9) node{\circle*{3.5}}; \draw(12.6,9.2) node{9};
              \draw(14.65,16) node{\circle*{3.5}}; \draw(14.5,16.5) node{16};
                \draw(16.15,14) node{\circle*{3.5}}; \draw(16.5,14.2) node{14};
                  \draw(17.65,6) node{\circle*{3.5}}; \draw(17.5,5.4) node{6};
                    \draw(19.15,11) node{\circle*{3.5}}; \draw(19,11.5) node{11};
                      \draw(20.65,8) node{\circle*{3.5}}; \draw(20.9,8.2) node{8};
                        \draw(22.15,5) node{\circle*{3.5}}; \draw(22,4.5) node{5};
                          \draw(23.65,10) node{\circle*{3.5}}; \draw(24,10) node{10};

\draw [thick]
(-0.5,16.5)--(1,12)--(2.5,7)--(4,1)--(5.5,3)--(7,13)--(8.5,15)--(10,2)--(11.5,4)--(13,9)--(14.5,16)--(16,14)--(17.5,6)--(19,11)--(20.5,8)--(22,5)--(23.5,10)--(25,16.5);
 \draw(-0.5,16.8) node{$+\infty$};
 \draw(25,16.8) node{$+\infty$};

\end{tikzpicture}
\end{center}
\caption{The group action $\Phi_S$ on  $\sigma=12\,7\,1\,3\,13\,15\,2\,4\,9\,16\,14\,6\,11\,8\,5\,10$ (the dotted lines in red indicate the movements of rlmin-double ascents and  lrmin-double descents).\label{new-action}}
\end{figure}
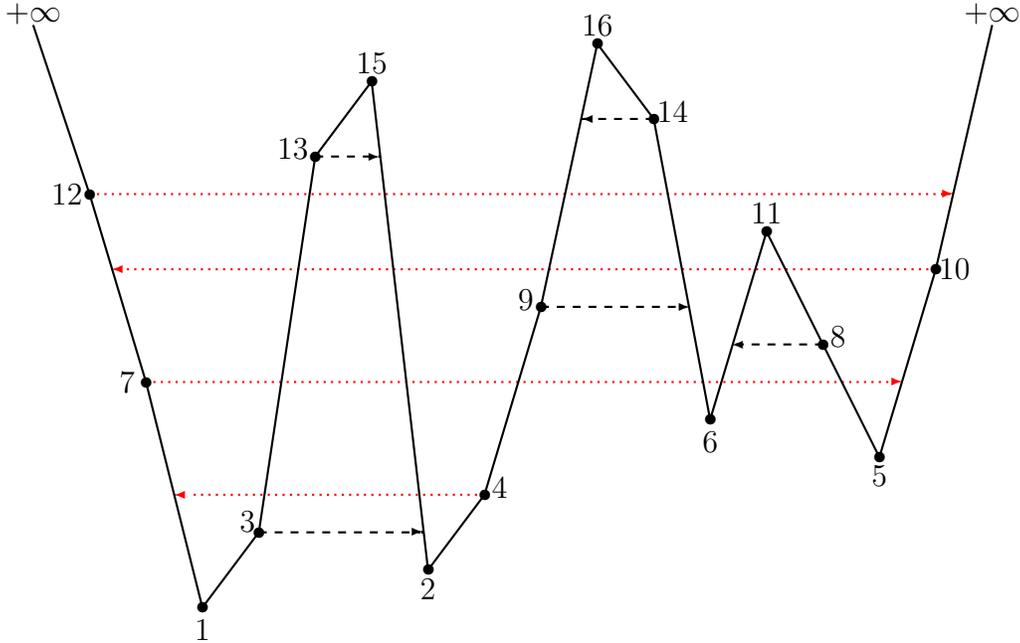

Now, for any $x\in[n]$ and $\sigma\in\S_n$, introduce the new action $\Phi_x$ by
$$
\Phi_x(\sigma):=
\begin{cases}
\sigma,& \text{if $x$ is a peak or a valley of $\sigma$;}\\
\varphi_x(\sigma),&\text{if $x$ is an internal double ascent or an internal double descent of $\sigma$};\\
\psi_x(\sigma), &\text{if $x$ is a rlmin-double ascent or a lrmin-double descent of $\sigma$}.
\end{cases}
$$
It is routine to check that all $\Phi_x$'s are involutions and commute. Therefore, for any $S\subseteq[n]$ we can define the function $\Phi_S:\S_n\rightarrow\S_n$ by $\Phi_S=\prod_{x\in S}\Phi_x$, where multiplication is the composition of functions. Hence the group $\Z_2^n$ acts on $\S_n$ via the function $\Phi_S$.
See Fig.~\ref{new-action} for a nice visualization of the group action $\Phi_S$ on permutations.

\subsection{A unified generalization of Theorems~\ref{mainthm2} and~\ref{Ji:gam}}

The main features of the action $\Phi_S$ lies in the following two lemmas, whose proofs are straightforward and will be omitted.

\begin{lem}\label{ddtoda}
If $x$ is an internal double ascent {\rm(}resp., internal double descent{\rm )} of $\sigma$, then $x$ becomes an internal double descent {\rm(}resp., internal double ascent{\rm)} of $\Phi_x(\sigma)$.
\end{lem}

\begin{lem}\label{lrmin}
If $x$ is a rlmin-double ascent {\rm(}resp.,  lrmin-double descent{\rm )} of $\sigma$, then $x$ becomes a lrmin-double descent {\rm(}resp., rlmin-double ascent{\rm )} of $\Phi_x(\sigma)$.
Consequently,
$$
\lmin(\sigma)+\rmin(\sigma)=\lmin(\Phi_x(\sigma))+\rmin(\Phi_x(\sigma)).
$$
\end{lem}

For $\Pi\subseteq\S_n$, we introduce the {\em generalized Stirling-Eulerian polynomial}
\begin{equation*}
A(\Pi;x,y|\a)=\sum_{\sigma \in \Pi} x^{{\rm des}(\sigma)}y^{{\rm asc}(\sigma)}{\alpha}^{{\rm {LRmin}}(\sigma)+{\rm RLmin}(\sigma)-2}.
\end{equation*}
We aim to prove the following common generalization of Theorems~\ref{mainthm2} and~\ref{Ji:gam}.

\begin{thm}\label{thm:PiP}
Suppose that $\Pi\subseteq\S_n$ for $n\geq1$ is invariant under the action $\Phi_S$ for any $S\subseteq[n]$. Then,
\begin{equation}\label{eq:PiP}
A(\Pi;x,y|\a)=\sum_{\sigma\in\Pi} (u_1u_2)^{{\rm M}(\sigma)}u_3^{{\rm da}(\sigma)}u_4^{\dd(\sigma)}\a^{{\rm LRmin}(\sigma)+{\rm RLmin}(\sigma)-2},
\end{equation}
where $u_3+u_4=x+y$ and $u_1u_2=xy$.
\end{thm}

\begin{proof}
For the sake of convenience, we denote ${\rm st}(\sigma)={\rm LRmin}(\sigma)+{\rm RLmin}(\sigma)-2$  for any $\sigma\in\S_n$.
Consider the group $\Z_2^n$ acts on $\Pi\subseteq\S_n$ via the function $\Phi_S$ and let $\Orb(\sigma)=\{g(\sigma): g\in\Z^n_2\}$ be the orbit of $\sigma\in\Pi$ under this action.
By Lemmas~\ref{ddtoda} and~\ref{lrmin}, we have
\begin{equation}\label{act:Phi1}
\sum_{\pi\in\Orb(\sigma)}(u_1u_2)^{{\rm M}(\pi)}u_3^{{\rm da}(\pi)}u_4^{\dd(\pi)}\a^{{\rm st}(\pi)}= (u_1u_2)^{{\rm M}(\bar\sigma)}(u_3+u_4)^{\da(\bar\sigma)}\a^{{\rm st}(\bar\sigma)},
\end{equation}
where $\bar\sigma$ denotes  the unique element in $\Orb(\sigma)$ that has no double descents. Applying the same argument and utilizing \eqref{rel-aa},  we deduce that
$$
\sum_{\pi\in\Orb(\sigma)}x^{{\rm des}(\pi)}y^{{\rm asc}(\pi)}\a^{{\rm st}(\pi)}= (xy)^{{\rm M}(\bar\sigma)}(x+y)^{\da(\bar\sigma)}\a^{{\rm st}(\bar\sigma)}.
$$
Comparing with~\eqref{act:Phi1} yields
$$
\sum_{\pi\in\Orb(\sigma)}x^{{\rm des}(\pi)}y^{{\rm asc}(\pi)}\a^{{\rm st}(\pi)}=\sum_{\pi\in\Orb(\sigma)}(u_1u_2)^{{\rm M}(\pi)}u_3^{{\rm da}(\pi)}u_4^{\dd(\pi)}\a^{{\rm st}(\pi)}
$$
when  $u_3+u_4=x+y$ and $u_1u_2=xy$.
Summing over all orbits of $\Pi$  then gives~\eqref{eq:PiP}.
\end{proof}

Taking $\Pi=\S_n$ in Theorem~\ref{thm:PiP} gives Theorem~\ref{Ji:gam}. As $\PRW_{n+1}$ is invariant under the action $\Phi_S$ for any $S\subseteq[n+1]$, taking $\Pi=\PRW_{n+1}$ in Theorem~\ref{thm:PiP} then gives Theorem~\ref{mainthm2}.

\begin{thm}\label{thm:Gamm}
Suppose that $\Pi\subseteq\S_n$ for $n\geq1$ is invariant under the action $\Phi_S$ for any $S\subseteq[n]$. Then, $A(\Pi;x,y|\a)$ admits the $\gamma$-positivity expansion
\begin{equation*}
A(\Pi;x,y|\a)=\sum_{\sigma \in \Pi\atop \dd(\sigma)=0}(xy)^{{\rm des}(\sigma)} (x+y)^{n-1-{\rm des}(\sigma)}\a^{{\rm {LRmin}}(\sigma)+{\rm RLmin}(\sigma)-2}.
\end{equation*}
\end{thm}
\begin{proof}
Setting $u_4=0$ in Theorem~\ref{thm:PiP} and using relationships~\eqref{dd:da} and~\eqref{M:des} yields the desired result.
\end{proof}

 \section*{Acknowledgement}

This work was supported by the National Science Foundation of China  and the project of Qilu Young Scholars of Shandong University.


\begin{thebibliography}{99}

\bibitem{And} D. Andr\'e, D\'eveloppement de $\sec x$ and $\tan x$, C. R. Math. Acad. Sci. Paris, {\bf88} (1879), 965--979.

\bibitem{Ath} C.A. Athanasiadis, Gamma-positivity in combinatorics and geometry, S\'em. Lothar. Combin. 77 (2018), Article B77i, 64pp (electronic).

\bibitem{br} P. Br\"and\'en, Actions on permutations and unimodality of descent polynomials, European J. Combin., {\bf29} (2008), 514--531.

\bibitem{br2} P. Br\"and\'en, Unimodality, log-concavity, real-rootedness and beyond, {\em Handbook of Enumerative Combinatorics}, CRC Press Book. \href{http://arxiv.org/abs/1410.6601}{(arXiv:1410.6601)}

\bibitem{Branden-Jochemko-2022} P. Br\"and\'en  and K.  Jochemko, The Eulerian transformation,  Trans. Amer. Math. Soc., {\bf375} (2022),  1917--1931.

\bibitem{Carlitz-Scoville-1974} L. Carlitz and R. Scoville, Generalized Eulerian numbers: combinatorial applications, J. Reine Angew. Math., {\bf265} (1974), 110--137.

\bibitem{chen} W.Y.C. Chen, Context-free grammars, differential operators and formal power series, Theor. Comput. Sci., {\bf 117} (1993), 113--129.



 \bibitem{Chen-Fu-2017} W.Y.C. Chen and A.M. Fu, Context-free grammars for permutations and increasing trees, Adv. in Appl. Math.,  {\bf82} (2017), 58--82.

  \bibitem{Chen-Fu-2022} W.Y.C. Chen and A.M. Fu,   A Context-free grammar for the $e$-positivity of the trivariate second-order Eulerian polynomials, Discrete Math., {\bf345}  (2022), Article 112661.



  \bibitem{Chen-Fu-2022-a} W.Y.C. Chen and A.M. Fu,  The Dumont ansatz for the Eulerian polynomials, peak polynomials and derivative polynomials, Ann. Comb., {\bf27} (2023), 707--735.

   \bibitem{Chen-Fu-Yan-2023}  W.Y.C. Chen, A.M. Chen and S.H.F. Yan, The Gessel correspondence and the partial $\gamma$-positivity of the Eulerian polynomials on multiset Stirling permutations, European J. Combin.,  {\bf109} (2023),  Article  103655.

 \bibitem{Chung-Graham-Knuth-2010} F. Chung, R. Graham and D. Knuth, A symmetrical Eulerian identity, J. Comb., {\bf1} (2010),  29--38.

\bibitem{Chung-Graham-2012} F. Chung and R. Graham, Generalized Eulerian Sums, J. Comb., {\bf3} (2012),  299--316.



\bibitem{fs} D. Foata and M.-P. Sch\"utzenberger, {\em Th\'eorie g\'eom\'etrique des polyn\^omes eul\'eriens},  Lecture Notes in Mathematics, Vol. 138, Springer-Verlag, Berlin, 1970.


\bibitem{fsh} D. Foata and V. Strehl, Rearrangements of the symmetric group and enumerative properties of the tangent and secant numbers, Math. Z., {\bf137} (1974), 257--264.


\bibitem{Haglund-Zhang-2019}  J. Haglund and P.B. Zhang, Real-rootedness of variations of Eulerian polynomials,  Adv. in Appl. Math., {\bf109} (2019), 38--54.

\bibitem{Han-Lin-Zeng-2011} G.-N. Han, Z. Lin and J. Zeng,  A symmetrical $q$-Eulerian identity,  S\'em. Lothar. Combin., Art. B67c, 11 pp.

\bibitem{Ji-2023} K.Q. Ji, The $(\alpha,\beta)$-Eulerian Polynomials and Descent-Stirling Statistics on Permutations, \href{https://arxiv.org/pdf/2310.01053.pdf}{arXiv:2310.01053.}

 \bibitem{Lin-2013} Z. Lin,  On some generalized $q$-Eulerian polynomials,  Electron. J. Combin., {\bf20} (2013),   \#P55.

\bibitem{lin} Z. Lin, Proof of Gessel's $\gamma$-positivity conjecture, Electron. J. Comb., 23 (2016), \#P3.15.

\bibitem{lmz} Z. Lin, J. Ma and P.B. Zhang,  Statistics on multipermutations and partial $\gamma$-positivity, J. Combin. Theory Ser. A, {\bf183} (2021), Article 105488.

\bibitem{lwz}  Z. Lin, D.G.L.Wang and J. Zeng,  Around the $q$-binomial-Eulerian polynomials, European J. Combin., {\bf78} (2019), 105--120.

\bibitem{lz} Z. Lin and J. Zeng,
The $\gamma$-positivity of basic Eulerian polynomials via group actions,
J. Combin. Theory, Ser.  A, {\bf135} (2015), 112--129.

\bibitem{mp} J. Ma and K. Pan,  $(M,i)$-multiset Eulerian polynomials, Adv. in Appl. Math., {\bf149} (2023), Article 102547.

 \bibitem{Ma-Ma-Yeh-2017} J. Ma, S.-M. Ma, and Y.-N. Yeh, Recurrence relations for binomial-Eulerian polynomials, \href{https://arxiv.org/abs/1711.09016}{arXiv:1711.09016}.

  \bibitem{Ma-Ma-Yeh-2019} S.-M. Ma, J. Ma and Y.-N. Yeh, $\gamma$-Positivity and partial $\gamma$-positivity of descent-type polynomials, J. Combin. Theory Ser. A, {\bf167} (2019), 257--293.

\bibitem{npt} E. Nevo, T.K. Petersen and B.E. Tenner,  The $\gamma$-vector of a barycentric subdivision, J. Combin. Theory Ser. A, {\bf118} (2011), 1364--1380.

\bibitem{pet} T.K. Petersen,  {\em Eulerian Numbers}. With a foreword by Richard Stanley.  Birkh\"auser Advanced Texts: Basler Lehrb\"ucher. Birkh\"auser/Springer, New York, 2015.

\bibitem{prw} A. Postnikov, V. Reiner and L. Williams, Faces of generalized permutohedra, Doc. Math., {\bf13} (2008), 207--273.


\bibitem{sw} J.  Shareshian and M.L. Wachs,  Gamma-positivity of variations of Eulerian polynomials, J. Comb., {\bf11} (2020),  1--33.

\bibitem{st} R.P. Stanley, {\em Enumerative combinatorics, Vol. 1}, Cambridge University Press, Cambriage, 1997.

\bibitem{Stembridge-1997} J.R. Stembridge, Enriched $P$-partitions, Trans. Amer. Math. Soc., {\bf349} (1997), 763--788.




\bibitem{Zhuang-2017}   Y. Zhuang, Eulerian polynomials and descent statistics, Adv. in Appl. Math., {\bf90} (2017),  86--144.



\end{thebibliography}
\end{document}